\documentclass[11pt,dvipdfmx,oneside]{amsart}

\usepackage{amssymb,amsmath,amsthm}
\usepackage{hyperref}
\usepackage{tikz}
\usepackage{enumitem}
\usepackage{ascmac}
\usepackage{geometry}[margin = 20mm]

\setlist[enumerate,1]{label=(\arabic*)}

\setlist[enumerate,2]{label=(\theenumi.\arabic*), ref=\theenumi.\arabic*}

\theoremstyle{definition}
\newtheorem{theorem}{Theorem}[section]
\newtheorem{remark}[theorem]{Remark}
\newtheorem{lemma}[theorem]{Lemma}
\newtheorem{definition}[theorem]{Definition}
\newtheorem{corollary}[theorem]{Corollary}

\newtheorem{proposition}[theorem]{Proposition}
\newtheorem{question}[theorem]{Question}

\newcommand{\GL}{\mathsf{GL}}

\newcommand{\RCAo}{\mathsf{RCA}_0}
\newcommand{\WKLo}{\mathsf{WKL}_0}
\newcommand{\ACAo}{\mathsf{ACA}_0}

\newcommand{\hyp}{\mathchar`-}

\newcommand{\T}{\mathrm{T}}

\newcommand{\RFN}{\mathsf{RFN}}

\newcommand{\M}{\mathcal{M}}

\newcommand{\NN}{\mathcal{N}}
\newcommand{\ACA}{\mathsf{ACA}}

\newcommand{\GN}[1]{\ulcorner #1 \urcorner}

\renewcommand{\P}{\mathsf{P}}
\newcommand{\Q}{\mathsf{Q}}
\newcommand{\W}{\mathrm{W}}
\newcommand{\dom}{\mathrm{dom}}
\renewcommand{\a}{\mathrm{a}}
\renewcommand{\H}{\mathrm{H}}

\DeclareMathOperator{\Th}{\mathrm{Th}}
\DeclareMathOperator{\Tr}{\mathrm{Tr}}
\renewcommand{\Pr}{\mathrm{Pr}}

\newcommand{\PA}{\mathsf{PA}}
\newcommand{\LogAx}{\mathrm{LogAx}}

\title{On some generalizations of G\"{o}del's second incompleteness theorem}
\author{YUDAI SUZUKI}
\address{National~Institute~of~Technology, Oyama~College, Nakakuki~771, Oyama, Tochigi, Japan}
\email{yudai.suzuki.research@gmail.com}
\author{KEITA YOKOYAMA}
\address{Mathematical~Institute, Tohoku~University, Aramaki~Aza-Aoba~6-3, Aoba-ku, Sendai, Miyagi, Japan}
\email{keita.yokoyama.c2@tohoku.ac.jp}

\begin{document}

\maketitle

\begin{abstract}
In this note, we give some generalizations of G\"{o}del's second incompleteness theorem and study their surroundings. We revisit it from two perspectives.

One perspective is the relationship between the definable complexity of a theory and unprovability of its soundness. We clarify the relationship between this perspective and induction axioms.
We also determine the logical strength of Craig's trick, which is important for studying the definability of a theory, from the point of view of reverse mathematics.

The other perspective is semantic incompleteness. 
The second incompleteness theorem may be seen as the unprovability of the existence of a model.
It is known that `model' is replaced with `$\omega$-model' or `$\beta_n$-model'.
We give a new and unified proof of the $\omega$-model and $\beta_n$-model versions of the incompleteness theorem.
\end{abstract}

\section{Introduction}
In this note, we give two generalizations of the second incompleteness theorem and study their surroundings.

  The first one focuses on the relationship between the definability and the soundness of a theory.
  In the context of reflection principles, the second incompleteness can be written as follows.
  \begin{theorem}[The second incompleteness theorem]
    Let $T$ be a $\Sigma_1$-definable extension of $\mathsf{PA}$. Then, 
    $T$ is $\Pi_1$-sound if and only if $T \not \vdash \RFN_{\Pi_1}(T)$.
  \end{theorem}

  An analogue of this theorem for $\Sigma_{k+1}$-definable theories is independently observed by Kikuchi-Kurahashi \cite[Theorem 5.6]{KikuchiKurahashi} and Chao-Seraji \cite{ChaoSeraji} as follows.
  \begin{theorem}\label{Kikuchi-Kurahashi}
    Let $T$ be a $\Sigma_{k+1}$-definable extension of $\mathsf{PA}$. 
    Then, $T$ is $\Pi_{k+1}$-sound if and only if $T \not \vdash \RFN_{\Pi_{k+1}}(T)$.
  \end{theorem}
  In this note, we give a refinement of this theorem.
  We consider extensions of $I\Sigma_n$ rather than $\mathsf{PA}$. By this refinement, we clarify the role of the induction axiom in this equivalence.
  We also give a second-order analogue of Theorem \ref{Kikuchi-Kurahashi}.

Our second generalization is an extension of the semantic incompleteness theorem.
By completeness theorem, the consistency of a theory is equivalent to the existence of a model of it. Thus, the second incompleteness theorem can be rephrased as follows for suitable theories.
\begin{quotation}
  If a theory $T$ has a model, then $T$ has a model of `$T$ has no model'.
\end{quotation}
In \cite{Steel}, Steel proved the following analogue of this formulation.
  \begin{theorem}[$\omega$-model incompleteness]
    Let $T$ be a $\Pi^1_1$-definable theory including $\ACA_0$.
    If $T$ has a coded $\omega$-model, then $T$ + `there is no coded $\omega$-model of $T$' also has a coded $\omega$-model.
  \end{theorem}

  For this theorem, Steel gave a purely recursion theoretic proof based on a well-foundedness property of Turing jumps. 
  On the one hand, in Simpson's monograph \cite[Theorem VIII.5.6.]{Simpson} a proof-theoretic proof is given for  $\Sigma^0_1$-definable theories. This proof is based on the second incompleteness theorem, so one needs the assumption that $T$ is $\Sigma^0_1$-definable.
In \cite{MummertSimpson}, Mummert and Simpson proved a $\beta_n$-model version of Steel's incompleteness theorem (for the details, see Theorem 2.1 and Remark 2.2 in \cite{MummertSimpson}).
\begin{theorem}[$\beta_n$-model incompleteness]
    Let $T$ be a $\Sigma^0_1$-definable theory including $\ACA_0'$.
    If $T$ has a coded $\beta_n$-model, then $T$ + `there is no coded $\beta_n$-model of $T$' also has a coded $\beta_n$-model.
  \end{theorem}
Since Mummert and Simpson followed the proof in \cite[Theorem VIII.5.6.]{Simpson}, they assumed $T$ to be $\Sigma^0_1$-definable. 
Later, Lutz and Walsh \cite{LutzWalsh,LutzWalshCorrigenda} extended this result to $\Sigma^1_n$-definable theories\footnote{In fact, they pointed out that their proof can be applied to $\mathsf{A}(\Sigma^1_n)$-definable theories.
Here, a formula is $\mathsf{A}(\Sigma^1_n)$ if it is constructed from $\Sigma^1_n$-formulas by disjunction, conjunction, negation and number quantifiers.}. 
Their proof was based on the well-foundedness of hyperjumps, just as Steel's proof was based on a well-foundedness property of Turing jumps. So these proofs are very similar, but not identical.
In this note, we give a general method to prove both  Steel's result and that of Lutz and Walsh.

\section*{Acknowledgement}
The authors would like to thank Taishi Kurahashi and Koshiro Ichikawa for their fruitful comments on this study.
The first author also would like to thank Taishi Kurahashi for introducing him to the paper \cite{KikuchiKurahashi}, which motivated the present research.
The work of the first author is partially supported by JSPS KAKENHI grant number 26K17028.
The work of the second author is partially supported by JSPS KAKENHI grant numbers JP21KK0045, JP23K03193 and 26K00615.


\section{Soundness, reflection and definability}
In this section, we study the relationship between the definability and the soundness of a theory.
We first prove that a $\Gamma$-definable $\check{\Gamma}$-sound theory does not prove its $\check{\Gamma}$-soundness where $\check{\Gamma}$ is the dual of $\Gamma$. Then, we show that the converse holds in a sense. We also consider the logical strength of Craig's trick, which plays an important role in the study of the definability.

The unprovability of $\check{\Gamma}$-soundness was first observed by Kikuchi-Kurahashi\cite{KikuchiKurahashi} and Chao-Seraji\cite{ChaoSeraji} for extensions of $\mathsf{PA}$. In \cite{Walsh,TowsnerWalsh}, it was observed for some theories of second-order arithmetic. In \cite{TowsnerWalsh}, a referee gave a simple and general proof for these results. We introduce it.

We present some definitions.  Throughout this note, $\GN{\varphi}$ denotes the G\"{o}del number of a formula $\varphi$. We do not distinguish a formula $\varphi$, its G\"{o}del number $\GN{\varphi}$ and its numeral $\overline{\GN{\varphi}}$ if it is clear from the context. For instance, we sometimes write $\sigma(\varphi)$ or $\sigma(\GN{\varphi})$ instead of $\sigma(\overline{\GN{\varphi}})$ for formulas $\sigma(x)$ and $\varphi$.

\begin{definition} 
  In the following definitions, we consider formulas and theories of first-order arithmetic.
  \begin{itemize}
    \item Let $\sigma(x)$ be a formula. Then $T_{\sigma}$ denotes $\{\varphi : \omega \models \sigma(\varphi)\}$. We call $T_{\sigma}$ the theory defined by $\sigma$.
    \item For a theory $T$, $\Th(T)$ denotes the set of theorems of $T$. That is, $\Th(T) = \{\varphi : T \vdash \varphi\}$.
    \item Let $T$ be a theory,  $\Gamma$ be a class of formulas and $\sigma \in \Gamma$. If $\Th(T_{\sigma}) = \Th(T)$ and $\sigma \in \Gamma$, then we call $T_{\sigma}$ a $\Gamma$-definition of $T$. If $T$ has a $\Gamma$-definition, then we say $T$ is $\Gamma$-definable.
    \item A theory $T$ is $\Gamma$-sound if any sentence $\sigma \in \Gamma$ provable from $T$ is true in the standard model $\omega$.
  \end{itemize}
  By replacing the standard model $\omega$ with the intended model $(\omega,2^{\omega})$ of second-order arithmetic, we can define the definability and the soundness of theories of second-order arithmetic.
\end{definition}

\begin{definition}\label{Def of ref}
  Let $\sigma(x)$ be a formula and $\Gamma \in \{\Sigma_{k+1}, \Pi_{k+1} : k \in \omega\}$. 
  Then, $\Gamma\hyp\RFN(T_{\sigma})$ denotes the formula stating that `$T_{\sigma}$ is $\Gamma$-sound'. Formally, 
  \begin{align*}
    \Gamma\hyp\RFN(T_{\sigma}) \equiv \forall x \in \Gamma(\Pr_{T_{\sigma}}(x) \to \Tr_{\Gamma}(x)).
  \end{align*}
  Here, $\Pr_{T_{\sigma}}$ is the provability predicate for $T_{\sigma}$, and $\Tr_{\Gamma}$ is the $\Gamma$-truth predicate.
\end{definition}

Henceforth, $\check{\Gamma}$ denotes the dual of $\Gamma$. That is, if $\Gamma=\Sigma^0_k$, then $\check{\Gamma}=\Pi^0_k$. If $\Gamma=\Pi^0_k$, then $\check{\Gamma} =\Sigma^0_k$.

\begin{lemma}
  Let $\Gamma \in \{\Sigma_{k+1},\Pi_{k+1}: k \in \omega\}$. If a theory $T$ is $\check{\Gamma}$-sound, then $T + \{\theta \in \Gamma : \omega \models \theta\}$ is consistent.
\end{lemma}
\begin{proof}
  For the sake of contradiction, assume that $T + \{\theta \in \Gamma : \omega \models \theta\}$ is inconsistent.
  Then there are $\theta_0,\ldots,\theta_{n-1} \in \Gamma$ such that $T + \{\theta_i : i < n\}$ is inconsistent. 
  Put $\theta \equiv \bigwedge_{i<n} \theta_i$. Then $\theta$ is still in $\Gamma$. In addition, $T \vdash \lnot \theta$ and $\omega \models \theta$. However, this is impossible because $T$ is $\check{\Gamma}$-sound.
\end{proof}

The following proof is pointed out by a referee in \cite[Theorem 2.4.]{TowsnerWalsh}.
\begin{theorem}\label{generalized incomp}
  Let $\Gamma \in \{\Sigma_{k+1},\Pi_{k+1}: k \in \omega\}$ and $\sigma(x) \in \Gamma$ such that $T_{\sigma} \supseteq I\Sigma_1$.
  If $T_{\sigma}$ is $\check{\Gamma}$-sound, then $T_{\sigma} \not \vdash \check{\Gamma}\hyp\RFN(T_{\sigma})$. 
\end{theorem}
\begin{proof}
  We may assume that $T_{\sigma}$ contains a formula $A_{I\Sigma_1}$ which axiomatizes $I\Sigma_1$.

  For the sake of contradiction, assume $T_{\sigma} \vdash \check{\Gamma}\hyp\RFN(T_{\sigma}) $. 
  Then, there are $\varphi_0,\ldots,\varphi_{n-1} \in T_{\sigma}$ such that $\varphi_0 \equiv A_{I\Sigma_1}$ and 
  $\{\varphi_i : i < n\} \vdash \check{\Gamma}\hyp\RFN(T_{\sigma})$.

  We note that $T_0 = \{\varphi_i : i < n\} \cup \{\sigma(\varphi_i) : i < n\}$ is consistent because $T_0$ is a subset of $T + \{\theta \in \Gamma : \omega \models \theta\}$. Therefore, $T_0$ is a finite consistent extension of $I\Sigma_1$. In particular, $T_0$ does not prove its consistency. 

  By definition of $T_0$, it believes that $\{\varphi_i : i < n\}$ is a subtheory of $T_{\sigma}$. Thus,
  within $T_0$, 
  \begin{align*}
    \Pr_{T_0}(0=1) &\to \Pr_{\varphi_0,\ldots,\varphi_{n-1}}(\lnot \sigma(\varphi_0) \lor \cdots \lor \lnot \sigma(\varphi_{n-1})) && (T_0 = \{\varphi_i : i < n\} \cup \{\sigma(\varphi_i) : i < n\}) \\
    &\to \Pr_{T_{\sigma}}(\lnot \sigma(\varphi_0) \lor \cdots \lor \lnot \sigma(\varphi_{n-1})) && (\varphi_i \in T_{\sigma} \text{ for all } i < n) \\
    &\to \lnot \sigma(\varphi_0) \lor \cdots \lor \lnot \sigma(\varphi_{n-1}) && (\check{\Gamma}\hyp\RFN(T_{\sigma})).
  \end{align*}
  Since $T_0 \vdash \sigma(\varphi_0) \land \cdots \land \sigma(\varphi_{n-1})$, $T_0 \vdash \lnot \Pr_{T_0}(0=1)$. This means that $T_0$ proves its own consistency, a contradiction.
\end{proof}

\begin{remark}
  By the same argument, we can prove the previous theorem for $\Gamma \in \{\Sigma^0_{k+1},\Pi^0_{k+1}: k \in \omega\}$ if $T_{\sigma} \supseteq \RCAo$, and for $\Gamma \in \{\Sigma^1_{k+1},\Pi^1_{k+1} : k \in \omega\}$ if $T_{\sigma} \supseteq \ACAo$. In the latter case, $\ACAo$ is needed for defining a truth predicate for $\check{\Gamma}$-formulas. Here, a $\Sigma^1_1$ truth predicate means a truth predicate for formulas of the form $\exists X \Pi^0_2$. For simplicity, we write $\Sigma^1_1\hyp\RFN$ to mean $(\exists X \Pi^0_2)\hyp\RFN$. We use a similar notation for other classes.
\end{remark}

As in the previous theorem, the soundness of a theory implies the unprovability of the soundness. Conversely, the unprovability of the soundness implies the soundness of a theory. 
It is proved in \cite[Theorem 5.6]{KikuchiKurahashi} that for any $\Sigma_{k+1}$-definable extension $T$ of $\PA$, the following are equivalent.
  \begin{enumerate}
    \item $T$ is $\Pi_{k+1}$-sound,
    \item for any $\Sigma_{k+1}$-definition $T_{\sigma}$ of $T$, $T \not \vdash \Pi_{k+1}\hyp\RFN(T_{\sigma})$,
    \item for any $\Sigma_{k+1}$-definition $T_{\sigma}$ of $T$, $T \not \vdash \RFN(T_{\sigma})$.
  \end{enumerate}
We give a more detailed observation.

\begin{lemma}\label{Char of Soundness}
  Let $T$ be a $\Sigma_{k+1}$-definable extension of $I\Sigma_1$. Then the following are equivalent.
  \begin{enumerate}
    \item $T$ is $\Pi_{k+1}$-sound,
    \item for any $\Sigma_{k+1}$-definition $T_{\sigma}$ of $T$, $T \not \vdash T_{\sigma} = \varnothing$.
  \end{enumerate} 
\end{lemma}
\begin{proof}
$(1) \Rightarrow (2)$ We show the contraposition. Assume there is a $\Sigma_{k+1}$-definiton $T_{\sigma}$ of $T$ such that $T \vdash T_{\sigma} = \varnothing$. Take a formula $\varphi \in T$.
  Then $\omega \models \sigma(\varphi)$ but $T \vdash \lnot \sigma(\varphi)$. Thus, $\lnot \sigma(\varphi)$ is a $\Pi_{k+1}$-formula provable in $T$ but not true.

 $(2) \Rightarrow (1)$ We show the contraposition. Assume $T$ is not $\Pi_{k+1}$-sound. 
 Let $\theta$ be a $\Pi_{k+1}$-formula such that $T \vdash \theta$ but $\theta$ is not true.
 Take a $\Sigma_{k+1}$-definiton $\sigma'(x)$ of $T$ and 
 put $\sigma(x) \equiv \sigma'(x) \land \lnot \theta$. Then
 \begin{itemize}
   \item $\sigma$ is $\Sigma_{k+1}$ because both of $\sigma',\lnot \theta$ are $\Sigma_{k+1}$.
   \item $\omega \models T_{\sigma'} = T$ because $\lnot \theta$ is true.
   \item $T \vdash T_{\sigma} = \varnothing$ because $T \vdash \theta$.
 \end{itemize}
 Thus, $T_{\sigma}$ is a $\Sigma_{k+1}$-definition of $T$ such that $T \vdash T_{\sigma} = \varnothing$.
\end{proof}

\begin{theorem}\label{Unprovability-Soundness}
  Let $T$ be a $\Sigma_{k+1}$-definable extension of $I\Sigma_1$. Assume there is an $n$ such that for any $\Sigma_{k+1}$-definition $T_{\sigma}$ of $T$, $T + \Pi_n\hyp\RFN(\varnothing) \not \vdash \Pi_n\hyp\RFN(T_{\sigma})$. Then, $T$ is $\Pi_{k+1}$-sound.
\end{theorem}
\begin{proof}
Let $n,k$ and $T$ be as above.
Assume $T$ is not $\Pi_{k+1}$-sound for the sake of contradiction. Then, there is a $\Sigma_{k+1}$-definiton $T_{\sigma}$ of $T$ such that $T \vdash T_{\sigma} = \varnothing$ by Lemma \ref{Char of Soundness}. Now $T + \Pi_n\hyp\RFN(\varnothing) \vdash \Pi_n\hyp\RFN(T_{\sigma})$, a contradiction.
\end{proof}

Let $T$ be a $\Sigma_k$-definable theory.
For $n \in \omega$, put $P_k(n)$ as the following condition.
\begin{quote}
  $P_k(n)$: for all $\Sigma_k$-definition $T_{\sigma}$ of $T$, $T \not \vdash \Pi_n\hyp\RFN(T_{\sigma})$.
\end{quote}

The following is a refinement of \cite[Theorem 5.6]{KikuchiKurahashi}.
\begin{corollary}
  Let $T$ be a $\Sigma_{k+1}$-definable extension of $I\Sigma_{1}$. 
  Then, for each $k$, we have
  \begin{align*}
     &P_{k+1}(1) \Rightarrow P_{k+1}(2) \Rightarrow  P_{k+1}(3) \Rightarrow \cdots  \\
     &P_{k+1}(n+3) \Rightarrow T \text{ is $\Pi_{k+1}$-sound}\;\; \text{  if $T \supseteq I\Sigma_{n+1}$} \\
     &T \text{ is $\Pi_{k+1}$-sound} \Rightarrow P_{k+1}(k+1)
  \end{align*}
\end{corollary}
\begin{proof}
  The implication $P_{k+1}(m) \Rightarrow P_{k+1}(m+1)$ is trivial.
  The implication $P_{k+1}(n+3) \Rightarrow T \text{ is }\Pi_
  {k+1}$-sound follows from because Theorem \ref{Unprovability-Soundness} because $I\Sigma_{n+1}$ is equivalent to $\Pi_{n+3}\hyp\RFN(\varnothing)$.
  The implication ($T \text{ is $\Pi_{k+1}$-sound} \Rightarrow P_{k+1}(k+1)$) is Theorem \ref{generalized incomp}.
\end{proof}

\begin{proposition}
  The implication $P_{k+1}(m) \Rightarrow P_{k+1}(m+1)$ is strict for $k \geq 1$ and $m \leq k$.
\end{proposition}
\begin{proof}
  Let $k \geq 2$. 
  We define theories $S,T$ and $U$ as follows. 
  \begin{align*}
    &S_0 = I\Sigma_1 + \Th_{\Pi_{m}}(\omega), \\
    &S_1 = I\Sigma_1 + \Pi_{m}\hyp\RFN(S_0), \\
    &S_2 = \Th(S_0) \cap \Th(S_1)
  \end{align*}
  We also define formulas $\sigma_0(x),\sigma_1(x),\sigma_2(x)$ by 
  \begin{align*}
    &\sigma_0(x) \equiv x = \GN{A_{I\Sigma_1}} \lor (x \in \Pi_m \land \Tr_{\Pi_m}(x)), \\
    &\sigma_{1}(x) \equiv x = \GN{A_{I\Sigma_1}} \lor x =\GN{\Pi_{m}\hyp\RFN(T_{\sigma_0})}, \\
    &\sigma_{2}(x) \equiv \exists p(\Pr_{\sigma_0}(p) \land p(|p|-1) = x) \land \exists q (\Pr_{\sigma_1}(q) \land q(|q|-1) = x)
  \end{align*}
  Here $A_{I\Sigma_1}$ axiomatizes $I\Sigma_1$.
  Then $\sigma_i$ defines $S_{i}$ for $i = 0,1,2$. More formally, $S_1$ is defined as $I\Sigma_1 + \Pi_m\hyp\RFN(T_{\sigma_0})$.
  
  We note that $\Pr_{\sigma_0}(p)$ is of the form $\Pi_{m}$, and hence $\sigma_2$ is of the form $\Sigma_{m+1}$.
 Since $m \leq k$, $S_2$ is $\Sigma_{k+1}$-definable by $\sigma_2$. In addition, $U$ is sound by definition.
Therefore, $P_{k+1}(m+1)$ holds for $S_2$ by Theorem \ref{generalized incomp}.

We show that $S_2$ proves $\Pi_{m}\hyp\RFN(T_{\sigma_{2}})$. 
Since $\sigma_1$ is a $\Sigma_0$-formula, $\Pi_{m}\hyp\RFN(T_{\sigma_1})$ is a true $\Pi_{m}$-sentence. Thus, $S_0 \vdash \Pi_{m}\hyp\RFN(T_{\sigma_1})$ and hence $S_0 \vdash \Pi_{m}\hyp\RFN(T_{\sigma_2})$.
On the one hand, $S_1 \vdash \Pi_{m}\hyp\RFN(T_{\sigma_0})$ by definition. Thus $S_1 \vdash \Pi_{m}\hyp\RFN(T_{\sigma_0})$ as well. 
So $U \vdash \Pi_{m}\hyp\RFN(T_{\sigma_0})$. This means that $P_{k+1}(m)$ does not hold.
\end{proof}

\begin{remark}
  The above example is essentially given by Niebergall \cite{Niebergall}.
\end{remark}

\begin{question}
  The above example $S_2$ is $\Sigma_{m+1}$-definable rather than $\Sigma_{k+1}$-definable.
  Is there a properly $\Sigma_{k+1}$-definable theory witnessing the strictness of $P_k(m) \Rightarrow P_k(m+1)$?
\end{question}

\begin{remark}
The assumption $T \supseteq I\Sigma_{n+1}$ is optimal for the implication $P(n+3) \Rightarrow`T \text{ is $\Pi_k$-sound'}$ if $k \geq n + 4$ in the following sense. 
Let $T = I\Sigma_{n} + \lnot I\Sigma_{n+1}$.
Then, $T \not \vdash \Pi_{n+3}\hyp\RFN(T)$ because $I\Sigma_1 \vdash I\Sigma_{n+1} \leftrightarrow \Pi_{n+3}\hyp\RFN(\varnothing)$. On the one hand, $T$ is not $\Pi_{n+4}$-sound because $\lnot I\Sigma_{n+1}$ is a false $\Pi_{n+4}$-sentence.
\end{remark}

In the previous discussion, we mainly consider $\Sigma_{k+1}$-definable theories.
Since $\Sigma_{k+1}$-definability is equivalent to $\Pi_{k}$-definability by Craig's trick, it seems that we do not need consider $\Pi_k$-definable theories. 
However, it is not clear that the $\Pi_k$-definition $\pi$ of a theory $T$ obtained by Craig's trick defines the same theory as the original $\Sigma_{k+1}$-definition $\sigma$ over a weak theory $T_0$. Thus, it may happen that $T_0$ does not prove $\RFN(T_{\pi}) \leftrightarrow \RFN(T_{\sigma})$.
In fact, a bounded principle is needed to ensure Craig's trick works as desired.


\begin{definition}
  Let $\sigma(x) \equiv \exists y \sigma'(x,y)$ be a $\Sigma_{k+1}$-formula.
  Define a formula $C_{\sigma}(z)$ by 
  \begin{align*}
    C_{\sigma}(z) \equiv \exists x,y \leq z (\sigma'(x,y) \land z = \bigwedge_{i < y} x).
  \end{align*}
\end{definition}

\begin{remark}
  Formally, $x \land y$ is a sequence of the form $(\land,x,y)$. Thus, 
  $\bigwedge_{i < y} x$ is a sequence of the form
  $ (\land,x,(\cdots (\land,x,(\land,x,x))) \cdots )$.
  Therefore, if we know that $z$ is of the form $\bigwedge_{i < y} x$, then we can compute $x$ and $y$ from $z$.
\end{remark}

\begin{lemma}
  The following assertions are equivalent over $I\Sigma_1$.
  \begin{enumerate}
    \item $B\Sigma_{k+1}$.
    \item For any $\sigma(x) \in \Sigma_{k+1}$, $\Th(T_{\sigma}) \subseteq \Th(T_{C_{\sigma}})$.
    \item The theorem schema $\{\Th(T_{\sigma}) \subseteq \Th(T_{C_{\sigma}}) : \sigma \in \Sigma_{k+1}\}$ where $\sigma$ runs over meta-formulas.
  \end{enumerate}
  In (2), $\sigma \in \Sigma_{k+1}$ is a formula coded in the base theory, rather than a meta-formula.
\end{lemma}
\begin{proof}
We prove $1 \to 2 \to 3 \to  1$.

$(1 \to 2)$
  Let $\sigma(x) \equiv \exists y \sigma'(x,y) \in \Sigma_{k+1}$.
Let $p$ be a proof over $T_{\sigma}$.
By bounded $\Pi_0$-comprehension, take the following $A'$.
\begin{align*}
  A' = \{i < |p| : p(i) \in \LogAx \lor \text{$p(i)$ is inferred from $p(0),\ldots, p(i-1)$}\}.
\end{align*}
Here, $p$ denotes the length of $p$, and $\LogAx$ denotes the set of logical axioms.
Let $A = \{i < |p| : i \not \in A'\}$. Then $A$ is the set of non-logical axioms appearing in $p$.
Now we have $\forall i < |p|  \exists y (i \in A \to \sigma'(p(i),y))$.
By $B\Sigma_{k+1}$, there exist a number $y$ and a function $f: A \to y$ such that $\forall i \in A  (\sigma'(p(i),f(i)))$.
By using $f$, we can replace $p(i)$ to the concatenation of $\langle \bigwedge_{j < f(i)} p(i)\rangle$ and a derivation of $p(i)$ from $\bigwedge_{j < f(i)} p(i)$  for all $i \in A$.
Then the resulting sequence is a proof over $T_{C_{\sigma}}$ proving the same sentence as $p$.

$(2 \to 3)$ Trivial.

$(3 \to 1)$ We prove $B\Pi_{k}$. Let $\tau(x,y)$ be a $\Pi_k$-formula. 
Let $u$ be such that $\forall x \leq u \exists y \sigma(x,y)$.
We show that $\exists v \forall x \leq u \exists y \leq v \tau(x,y)$.

For each number $x$, let $a_x$ be the formula $0 < \overline{x}$.
Then each $a_x$ is independent from others over logical axioms.

Define $\sigma(z) \equiv \exists x \leq z(z = a_x \land \exists y \tau(x,y) )  $.
Then $\sigma$ is $\Sigma_{k+1}$.
Let $p$ be the sequence $\langle a_0, \ldots, a_{u}, a_0 \land \cdots \land a_{u} \rangle$. 
Then $p$ is a proof\footnote{More formally, we add a derivation of $a_0 \land \cdots \land a_u$ from $a_0,\ldots,a_u$ if it is needed.} of $a_0 \land \cdots \land a_{u}$ from $T_{\sigma}$.
Let $q$ be a proof of $a_0 \land \cdots \land a_u$ from $T_{C_{\sigma}}$. Let $A$ be the non-logical axioms appearing in $q$.
Since each member of $T_{C_{\sigma}}$ is of the from $a_x \land a_x \land \cdots \land a_x$ and each $a_x$ is independent from other $a_y$'s, $A$ must include formulas of the form $a_x \land \cdots \land a_x$ for all $x \leq u$.
Thus we have $\forall x \leq u \exists y \leq p \tau(x,y)$.
\end{proof}

\begin{corollary}
  The following assertions are equivalent over $I\Sigma_1$.
  \begin{enumerate}
    \item $B\Sigma_{k+1}$.
    \item Uniform $\Sigma_{k+1}$ Craig's trick: $\forall \sigma(x) \in \Sigma_{k+1}(\Th(T_{\sigma}) = \Th(T_{C_{\sigma}}))$.
    \item $\Sigma_{k+1}$ Craig's trick: The theorem schema $\{\Th(T_{\sigma}) = \Th(T_{C_{\sigma}}) : \sigma \in \Sigma_{k+1}\}$ where $\sigma$ runs over meta-formulas.
  \end{enumerate}
\end{corollary}
\begin{proof}
  It is enough to prove that over $B\Sigma_{k+1}$, $\forall \sigma \in \Sigma_{k+1} (\Th(T_{\sigma}) \supseteq \Th(T_{C_{\sigma}}))$ holds.

  We work in $B\Sigma_{k+1}$. Let $\sigma(x) \equiv \exists y \sigma'(x,y)$ be a $\Sigma_{k+1}$-formula.
  Let $p$ be a proof over $T_{C_{\sigma}}$ and $A$ be the set of non-nolgical axioms appearing in $p$. 
  By definition of $C_{\sigma}$, $\forall i \in A \exists x,y < p (\sigma'(x,y) \land p(i) = \bigwedge_{j < y} x )$.
  We note that these $x$ and $y$ are computable from $i \in A$.
Thus there is a function $f,g: A \to p$ such that $\forall i \in A(p(i) = \bigwedge_{j < g(i)} f(i) \land \sigma'(f(i),g(i)))$.
  By using $f,g$, we can transform $p$ to a proof over $T_{\sigma}$ by replacing each $p(i)$ to the concatenation of $\langle f(i) \rangle$ and the derivation of $\bigwedge_{j < g(i)} f(i) \rangle$ from $f(i)$.
\end{proof}

\begin{remark}
  The inclusion $\Th(T_{\sigma}) \supseteq \Th(T_{C_{\sigma}})$ holds already in $I\Sigma_1$ by the same proof as above.
\end{remark}

\begin{theorem}[$B\Sigma_{k+1}$]
$\forall \sigma \in \Sigma_{k+1}\exists \pi \in \Pi_{k} (\Th(T_{\sigma}) = \Th(T_{\pi}))$.
\end{theorem}
\begin{proof}
We work in $B\Sigma_{k+1}$.
Let $\sigma \in \Sigma_{k+1}$. 
Then $C_{\sigma}(z)$ is of the form $\exists x \leq z \pi'(x,z) $ for some $\pi' \in \Pi_{k}$.
Note that $\forall z(C_{\sigma}(z) \leftrightarrow \exists x \leq z \Tr_{\Pi_k}(\GN{\pi'},x,z))$. 
Since the latter is equivalent to a $\Pi_k$-formula, this completes the proof.
\end{proof}

We next consider the relationship between the soundness and the definability for theories of second-order arithmetic.
As in Theorem \ref{Unprovability-Soundness}, the reflection principle for the empty set is important.
\begin{lemma}[{\cite[Theorem 5.1.]{Frittaion}}]
  Over $\ACAo$, 
  \begin{itemize}
    \item $I\Pi^1_{n+1}$ is equivalent to $\Pi^1_{n+3}\hyp\RFN(\varnothing)$,
    \item $(I\Pi^1_{n+1})^{-}$ is equivalent to $\Sigma^1_{n+2}\hyp\RFN(\varnothing)$.
  \end{itemize}
  Here, $(I\Pi^1_{n+1})^{-}$ is the induction schema for $\Pi^1_{n+1}$ formulas without set parameters.
\end{lemma}

\begin{remark}
  In \cite{Frittaion}, the axiom of uniform reflection is used. On the other hand, we adopt global reflection as the definition of the reflection principle (Definition \ref{Def of ref}).  We note that the principles of global and uniform reflection are equivalent if the base theory is strong enough to control the truth predicates such as $\ACAo$.
\end{remark}

\begin{theorem}\label{Unprovability-Soundness-SecondOrder}
  Let $T \supseteq \ACAo$.
  \begin{enumerate}
    \item Assume $T \vdash I \Pi^1_{n+1}$ and $\Sigma^1_{k+1}$-definable. If $T \not \vdash \Pi^1_{n+3}\hyp\RFN(T_{\sigma})$ for any $\Sigma^1_{k+1}$ definition $T_{\sigma}$ of $T$, then $T$ is $\Pi^1_{k+1}$-sound.
    \item Assume $T \vdash (I\Pi^1_{n+1})^{-}$ and $\Pi^1_{k+1}$-definable.
    If $T \not \vdash \Sigma^1_{n+2}\hyp\RFN(T_{\sigma})$ for any $\Sigma^1_{k+1}$ definition $T_{\sigma}$ of $T$, then $T$ is $\Sigma^1_{k+1}$-sound.
  \end{enumerate}
\end{theorem}
\begin{proof}
  The same proof as in Theorem \ref{Unprovability-Soundness} works.
\end{proof}
\begin{corollary}
  Let $T \supseteq \ACAo$.
\begin{enumerate}
  \item Assume $T \vdash I \Pi^1_{n+1}$ and is $\Sigma^1_{k+1}$-definable. Then $T$ is $\Pi^1_{k+1}$-sound if and only if $T \not \vdash \Pi^1_{n+3}\hyp\RFN(T_{\sigma})$ for any $\Sigma^1_{k+1}$-definition $T_{\sigma}$ of $T$.
  \item Assume $T \vdash (I\Pi^1_{n+1})^{-}$ and is $\Pi^1_{k+1}$-definable. Then $T$ is $\Sigma^1_{k+1}$-sound if and only if $T \not \vdash \Sigma^1_{n+2}\hyp\RFN(T_{\sigma})$ for any $\Pi^1_{k+1}$-definition $T_{\sigma}$ of $T$.
  \end{enumerate}
\end{corollary}
\begin{proof}
  They are immediate from Theorem \ref{generalized incomp} and Theorem \ref{Unprovability-Soundness-SecondOrder}.
\end{proof}

\begin{remark}
  As in Theorem \ref{Unprovability-Soundness}, the assumption $T$ proves sufficiently strong induction is essential for Theorem \ref{Unprovability-Soundness-SecondOrder}. 

  Let $T = \ACAo + \lnot (I\Pi^1_{n+1})^{-} = \ACAo + \lnot \Sigma^1_2\hyp\RFN(\varnothing)$. Since $\lnot \Sigma^1_2\hyp\RFN(\varnothing)$ is of the form $\exists x \Pi^1_{2}$, there is a $\Pi^1_2$ formula $\psi$ such that 
  \begin{align*}
    \ACAo \vdash \lnot \Sigma^1_2\hyp\RFN(\varnothing) \to \psi \text{ and } (\omega,2^{\omega}) \models \psi \to \lnot \Sigma^1_2\hyp\RFN(\varnothing).
  \end{align*}
  Thus, $T$ does not prove $\Sigma^1_2\hyp\RFN(T_{\sigma})$ for any $\sigma$ but $T$ prove a false $\Pi^1_2$-sentence $\psi$.
\end{remark}

\section{omega-model reflections}

Roughly speaking, the second incompleteness theorem says that for any suitable theory $T$, if $T$ is consistent, then $T$ does not prove its own consistency. If $T$ is sufficiently strong to prove the completeness theorem (e.g., $T$ includes $\WKLo$), then the second incompleteness theorem can be rephrased as follows.
\begin{center}
  If $T$ has a model, then $T$ has a model of `$T$ has no model'.
\end{center}
We call this formulation a semantic incompleteness theorem.
For the semantic incompleteness theorem, it is known that `a model' can be replaced with `an $\omega$-model' \cite{Simpson, Steel} or `a $\beta_n$-model' \cite{LutzWalsh,LutzWalshCorrigenda,MummertSimpson}.
In this section, we introduce a more general form of the semantic incompleteness theorem including these results.

The second incompleteness theorem deeply depends on the properties of the provability predicate known as derivability conditions. 

\begin{definition}
   Let $\Box x$ be a formula whose free variable is just $x$.
  In what follows we write $\Box \varphi$ for $\Box \overline{\GN{\varphi}}$.
  Let $T$ be a recursive extension of $\mathsf{PA}$.
  We introduce properties D1, D2 and D3 for $\Box x$ as follows.
   \begin{description}
    \item [(D1)] If $T \vdash \varphi$, then $T \vdash \Box \varphi$,
    \item [(D2)] $T \vdash \Box \varphi \land \Box(\varphi \to \psi) \to \Box\psi$     
    \item [(D3)] $T \vdash \Box \varphi \to \Box\Box \varphi$.
  \end{description}
  We call D1, D2 and D3 derivability conditions.
\end{definition}

It is known that the second incompleteness theorem is obtained from derivability conditions as follows.
\begin{theorem}[Formalized L\"{o}b's theorem] \label{Generalized Lob}
  Assume that $\Box x$ satisfies derivability conditions.
  Then, $T \vdash \Box(\Box \varphi \to \varphi) \to \Box \varphi$ for any $\varphi$. 
\end{theorem}

Taking $\varphi$ to be $\bot$, a formula representing a contradiction (e.g., $0 = 1$), we obtain the following formalized second incompleteness theorem.
\begin{theorem}[Formalized second incompleteness theorem] 
  Assume that $\Box x$ satisfies derivability conditions.
  Then, $T \vdash \lnot \Box \bot \to \lnot \Box \lnot \Box \bot$.
\end{theorem}

In the context of the semantic incompleteness theorem, we interpret $\Box \varphi$ as `$\varphi$ is true in any model with some specific conditions'. For example, if $\Box \varphi$ is `$\varphi$ is true in any model of $T$', then $\lnot \Box \bot \to \lnot \Box \lnot \Box \bot$ means `if $T$ has a model, then $T$ has a model of ($T$ has no model)'. So this is the usual second incompleteness theorem.
We will see that both of Steel's incompleteness theorem for $\omega$-models \cite{Steel} and Lutz-Walsh's incompleteness theorem for $\beta_n$-models \cite{LutzWalsh,LutzWalshCorrigenda} can be seen as variants of this formulation.

\begin{definition}[$\ACAo$]
  Let $\M$ be a coded $\omega$-model and $x$ a code of a formula.
  We write $\M \models x$ if $\M$ is a weak-model of $x$, that is, there is an evaluation $f$ for $(\M,x)$ such that $f(x) = \top$. For the details, see also \cite[VII.2.1]{Simpson}. 
  For a formula $\sigma(x)$, we write $\M \models T_{\sigma}$ if $\forall x (\sigma(x) \to \M \models x)$.
\end{definition}
 
\begin{lemma}
  Let $T_{\sigma}$ be a theory defined by $\sigma(x) \in \Pi^1_{n+1}$.
  Then, the following is provable in $\ACA_0$. Let $\M$ be a coded $\omega$-model and $\M'$ be a coded $\beta_n$-model such that $\M \in \M'$.
  If $\M' \models [\M \models T_{\sigma}]$, then $\M \models T_{\sigma}$. 
\end{lemma}
\begin{proof}
  We work in $\ACAo$. Assume $\M' \models [\M \models T_{\sigma}]$. We will show that $\M \models T_{\sigma}$. That is, $\M \models x$ for any code $x$ of a formula with $\sigma (x)$.
  
  Let $x$ be a code of a formula such that $\sigma(x)$. Since $\sigma$ is $\Pi^1_{n+1}$ and $\M'$ is a $\beta_n$-model, $\M' \models \sigma(x)$. By the assumption $\M' \models [\M \models T_{\sigma}]$, we obtain $\M' \models [\M \models x]$.
  Therefore, there is an $f \in \M'$ such that $\M'$ believes `$f$ is an evaluation for $(\M,x)$ such that $f(x) = \top$'. Since the condition in the quotation marks is arithmetic, it is actually true. Therefore $\M \models x$ actually holds.
\end{proof}

\begin{lemma}[$\ACAo$]
  Let $\M'$ be a $\beta_n$-model and $\M \in \M'$.
  If $\M' \models [\M \text{ is a $\beta_n$-model}]$,
  then $\M'$ is a $\beta_n$-model.
\end{lemma}
\begin{proof}
  Let $\theta$ be a $\Pi^1_{n}$ formula with parameters from $\M$.
  Then, $\theta \leftrightarrow \M' \models \theta \leftrightarrow \M \models \theta$.  Thus we obtain $\theta \leftrightarrow \M \models \theta$. This completes the proof.
\end{proof}

\begin{lemma} \label{box and beta models}
  Let $T_{\sigma}$ be a $\Pi^1_{n+1}$-definable theory defined by $\sigma$. 
  Let $C$ be a new constant symbol.
  Define a formula $\Box x$ by [$x$ is a code of a formula and any coded $\beta_n$-model of $T_{\sigma}$ containing $C$ satisfies $x$].
  Then, $\Box x$ satisfies the derivability conditions over $\ACAo$.
\end{lemma}
\begin{proof}
  We will show that (D1), (D2) and (D3) holds.
  \begin{description}
    \item [(D1)] Assume $\ACAo \vdash \varphi$. 
    We show that $\ACAo \vdash \Box \varphi$.
    We work in $\ACAo$ and prove that any $\beta_n$-model of $T_{\sigma}$ satisfies $\varphi$.
    
    Let $p$ be a proof of $\varphi$ over $\ACAo$. Then $\ACAo$ proves that [$p$ is a proof of $\varphi$ over $\ACAo$]. Hence, $\ACAo$ proves that if $\M$ is a (weak-)model of $\ACAo$, then $\M \models \varphi$.
    Since [any $\beta_n$-model is a (weak-)model of $\ACAo$] is provable in $\ACAo$, $\Box \varphi$ is also provable in $\ACAo$.
    \item [(D2)] This is trivial.
    \item [(D3)] We work in $\ACAo$.
    Assume $\Box \varphi$ holds. 
    Let $\M$ be a coded $\beta_n$-model of $T$ containing $C$. We show that $\M \models \Box \varphi$. 

    Let $\M' \in \M$ such that  $\M \models [\M' \models T]$ and $C \in \M'$.
    Then $\M'$ is a $\beta_n$-model of $T$ by previous lemmas and hence $\M' \models \varphi$. Since $\M' \models \varphi$ is an arithmetical condition, $\M \models [\M' \models \varphi]$.
  \end{description}
      This completes the proof.
\end{proof}

\begin{theorem}[{Formalized $\beta_n$-incompleteness theorem}]\label{Formalized incomp}
  Let $T$ be a $\Pi^1_{n+1}$-definable theory. Then, 
  the following is provable in $\ACAo$ for any set $X$.
  \begin{quotation}
    If $T$ has a coded $\beta_{n}$-model containing $X$, then $T$ + [$T$ has no coded $\beta_n$-model] has a coded $\beta_n$-model containing $X$.
  \end{quotation}
\end{theorem}
\begin{proof}
  Let $\Box x$ as in the previous lemma. Then, $\Box x$ satisfies derivability conditions. Thus, 
  $\ACAo \vdash \lnot \Box \bot \to \lnot \Box \lnot \Box \bot$. This completes the proof.
\end{proof}

\begin{remark}
  In \cite{MummertSimpson}, it is stated that the previous theorem can be proved in $\ACA_0'$ if $T$ is $\Sigma^0_1$-definable.  Our result extends this theorem. The formalized version of the $\beta_n$-incompleteness theorem is not explicitly argued in \cite{LutzWalsh}, but it is not hard to formalize their proof in $\ACAo$.
\end{remark}

\begin{remark}
  As we have mentioned before, the proofs in Simpson's monograph \cite{Simpson} and the paper by Mummert and Simpson \cite{MummertSimpson} are based on the usual incompleteness theorem.
  Thus, in these proofs, $T$ is assumed to be $\Sigma^0_1$-definable. If we use Theorem \ref{generalized incomp} instead of the usual incompleteness theorem, the same proof works for Theorem \ref{Formalized incomp}.
\end{remark}

As stated in Theorem \ref{Formalized incomp}, the $\beta_n$-incompleteness theorem can be formalized in $\ACAo$.
Since a coded $\beta_n$-model is a model of $\ACAo$, we can use it in coded $\beta_n$-models.
In this way, we can obtain a stronger result.
\begin{corollary}
  Let $T$ be a $\Pi^1_{n+1}$-definable theory.
  Then, the following is provable in $\ACAo$ for any set $X$.
  \begin{quotation}
  If $\M$ is a coded $\beta_n$-model of $T$ containing $X$, then we can find an $\M$-computable $\beta_n$-model $\M'$ of $T$ satisfying [$T$ has no coded $\beta_n$-model containing $X$] and containing $X$.
  \end{quotation}
\end{corollary}
\begin{proof}
  Let $\Box x$ be as in the previous theorem.
  In the following argument, we work in $\ACAo$.
  
  Let $\M$ be a $\beta_n$-model of $T$ containing $X$. We show that there exists $\M' \leq_{\T} \M$ such that 
  $\M'$ is a model of `$T$ + there is no coded $\beta_n$-model of $T$'. By the previous theorem, $\ACAo \vdash \lnot \Box \bot \to \lnot \Box \lnot \Box \bot$.
  By (D1), we obtain $\Box(\lnot \Box \bot \to \lnot \Box \lnot \Box \bot)$ and hence $\M \models \lnot \Box \bot \to \lnot \Box \lnot \Box \bot$. Now there are two possibilities: $\M \models \Box \bot$ or $\M \models \lnot \Box \bot$.

Assume $\M \models \Box \bot$. Then $\M$ itself is a $\beta_n$-model of `$T$ + there is no coded $\beta_n$-model of $T$'.

Assume $\M \models \lnot \Box \bot$. Then $\M \models \lnot \Box \lnot \Box \bot$. Hence, there is a $\beta_n$-model $\M' \in \M$ of $T$ such that $\M' \models \Box \bot$. This $\M'$ is a $\beta_n$-model of `$T$ + there is no coded $\beta_n$-model of $T$'.
\end{proof}

\begin{remark}
  For a $\beta_n$-model $\M$ and a $\Pi^1_{n+1}$-definable theory $T_{\sigma}$, $\M \models \text{`there is no $\beta_n$-model of $T_{\sigma}$'}$ does not mean there is no $\M' \in \M$ such that $\M' \models T_{\sigma}$. This is because the theory ${T_{\sigma}}^{\M} = \{x : \M \models \sigma(x)\}$ which $\M$ regards as $T_{\sigma}$, may be larger than the actual $T_{\sigma}$. On the other hand, if $\sigma$ has sufficiently low complexity so that $\sigma$ is absolute for $\M$, then $\M \models \text{`there is no $\beta_n$-model of $T_{\sigma}$'}$ is equivalent to the condition that there is no $\beta_n$-model $\M'$ of $T_{\sigma}$  such that $\M' \in \M$. Thus, the following corollary holds.
\end{remark}

\begin{corollary}
Let $T$ be an $A(\Pi^1_n)$-definable theory. Here $A(\Pi^1_n)$ is the class of formulas constructed by $\lnot, \land, \forall x$ from $\Pi^1_n$ formulas. Then, the following is provable in $\ACAo$.
  \begin{quotation}
    If $\M$ is a coded $\beta_n$-model of $T$, then we can find an $\M$-computable $\beta_n$-model $\M'$ of $T$ such that no $\M'' \in \M$ is a $\beta_n$-model of $T$.
  \end{quotation}
\end{corollary}

\begin{remark}
  The notions of $\omega$-models of $\ACAo$ and $\beta_0$-models are the same. Thus, our result includes Steel's incompleteness for $\omega$-models.
\end{remark}

In \cite{Steel}, Steel proved a kind of well-foundedness property of Turing degrees, which implies a well-founded property of coded $\omega$-models. Using this result, he obtained his incompleteness theorem. In \cite{LutzWalsh}, Lutz and Walsh used a similar strategy. They proved the well-foundedness of hyperdegrees within $\ACAo$, and proved their incompleteness theorem using the well-foundedness of $\beta$-models obtained from that of hyperdegrees. 
We point out that starting from the well-foundedness of $\beta$-models, we can obtain a somewhat simpler proof for the well-foundedness of hyperdegrees. We note that the following proof is based on essentially the same idea as the proof of \cite[Theorem 1.1]{LutzWalsh}.

\begin{theorem}[$\ACAo$]\label{well-foundedness of hyperdegrees}
  There is no sequence of $\beta$-models $\langle \M_i \rangle_i$ such that $\M_0 \ni \M_1 \ni \M_2 \ni \cdots$
\end{theorem}
\newcommand{\DS}{\mathsf{DS}}
\begin{proof}
  Let $\DS'$ be the statement that there is a sequence of  $\beta$-models $\langle \M_i \rangle_i$ such that $\forall i (\M_i \ni \M_{i+1})$. We show that $\ACAo + \DS'$ proves the consistency of $\ACAo + \DS'$. Then  $\ACAo \vdash \lnot \DS'$ because $\ACAo + \DS'$ is inconsistent. In the following, we work in $\ACAo + \DS'$. Let $\langle \M_i \rangle_i$ be a witness of $\DS'$.

  Let $\beta(X)$ be the formula stating that $X$ is a $\beta$-model.
  Now we have $\exists \langle \NN_i \rangle_i \forall i (\NN_{i+1} \in \NN_{i} \in \M_1 \land \M_1 \models \beta(\NN_i)  )$ by taking $\langle \NN_i \rangle_{i} = \langle M_{i+2} \rangle_i$. Since $\M_0$ is a $\beta$-model, we obtain $\M_0 \models \exists \langle \NN_i \rangle_i \forall i (\NN_{i+1} \in \NN_{i} \in \M_1 \land \M_1 \models \beta(\NN_i)  ) $. 
  Then, the sequence $A = \langle \M_1 \rangle \ast \langle \NN_i \rangle_i$, the concatenation of $\langle M_1 \rangle$ and $\langle \NN_i \rangle_i$, exists in $\M_0$. Them $\M_0$ believes $A$ is a descending sequence of $\beta$-models because  $\M_1$ is a $\beta$-submodel of $\M_0$ and each $\NN_i$ is a $\beta$-submodel of $\M_1$. Thus, $\M_0$ is a model of $\ACAo +  \DS'$.
\end{proof}

The following is Theorem 1.1 in \cite{LutzWalsh}.
\begin{corollary}[$\ACAo$]
  The hyperdegrees are well-founded. That is, there is no sequence $\langle X_i \rangle_i$ such that 
  $\forall i \exists H(H = \mathcal{O}^{X_{i+1}} \land H \leq_{\H} X_{i}) $.
\end{corollary}
\begin{proof}
  Assume such a sequence $\langle X_i \rangle_i$ exists.
  Then, there is a $\mathcal{O}^{X_1}$-computable $\beta$-model $\M_1 \ni X_1$. We have \begin{align*}
    &(1)\;\exists \M_2(X_2 \in \M_2 \in \M_1 \land \M_1\models \beta(\M_1)), \\
    &(2)\;\forall \NN \forall i > 1 (X_i \in \NN \in \M_1 \land \M_1 \models \beta(\NN) \to \exists \NN' (X_{i+1} \in \NN' \in \NN \land \M_1 \models \beta(\NN'))).
  \end{align*} 
  (1) Note that $\mathcal{O}^{X_{2}} \in \M_1$ because $\mathcal{O}^{X_2} \leq_{\H} X_1 \in \M_1$. Thus, there exists a $\beta$-model $\M_2 \leq_{\T} \mathcal{O}^{X_2}$ such that $X_2 \in \M_2$. Then $\M_2 \in \M_1$ because $\mathcal{O}^{X_2} \in \M_1$. (2) is proved a similar argument.

  By (1) and (2), we can find a sequence of $\beta$-model $\langle M_{i+2} \rangle_i \leq_{\T} (\M_1 \oplus \langle X_i \rangle_i)^{(k)}$ for sufficiently large $k$ such that $\forall i(\M_{i+1} \ni \M_i)$ where $\bullet^{(k)}$ denotes the $k$-th Turing jump of $\bullet$. This contradicts Theorem \ref{well-foundedness of hyperdegrees}.
\end{proof}

\begin{corollary}[$\ACAo$]
  Let $\M$ be a coded $\beta$-model. Define a Kripke frame $F = (W,R)$ by $W = \{i : \M \models \beta(\M_i) \}$ where $\beta(X)$ is the formula asserting `$X$ is a $\beta$-model', and $i R j \leftrightarrow \M_i \ni \M_j$ for $i,j \in W$.   
  Then, $R$ is converse well-founded by Theorem \ref{well-foundedness of hyperdegrees}. In addition, it is easy to see that $R$ is transitive. Thus, $F$ is a $\GL$-frame.
\end{corollary}

As in \cite[Theorem 4.2]{LutzWalsh}, the well-foundedness of hyperdegrees derives a variant of semantic incompleteness theorem. We point out that \cite[Theorem 4.2]{LutzWalsh} can be formalized as follows.
\begin{corollary}[$\ACA_0^+$]
   Let $n > 0$ and $T$ be a theory. If $T$ has a $\beta_n$-model, then there is a $\beta_n$-model $\M$ of $T$ such that no $X \in \M$ is a $\beta_n$-model of $T$.
\end{corollary}
\begin{proof}
  Assume $T$ has a $\beta_n$-model $\M_0$ but $\forall \M(\M \models_{\beta_n} T \to \exists \M' \in \M( \M' \models_{\beta_n} T))$. 
  Then, the latter assumption can be rewritten as  $\forall \M(\M \models_{\beta_n} T \to \exists \M' \in \M(\M' \models T \land \M \models \beta(\M')))$.
  Thus we can find a sequence $\langle M_{i+1} \rangle_i \leq_{\a} \M^{(\omega)} \oplus T$ so that $\forall i(\M_0 \models \beta(\M_{i+1})) \land \forall i(\M_0 \models [\M_{i+2} \in \M_{i+1}])$. This contradicts Theorem \ref{well-foundedness of hyperdegrees}.
\end{proof}

We note that $\M^{(\omega)}$ is needed to decide whether $\M_i$ is a model of $T$. If $T$ is definable as a subclass of $\M$, then $\M' \models T$ can be seen as a property in $\M$. So in this case, $\ACA_0'$ is sufficient as the base theory. If $n$ is standard in addition, then $\ACAo$ is sufficient.

We next give an application of the $\beta_n$ incompleteness theorem in the context of computability theory.
In the context of reverse mathematics, it is usual to separate two theories by their consistency strength. In particular, we sometimes use the following argument to separate two theories $T$ and $S$: $T$ proves the existence of an $\omega$-model $S$. Then $T$ proves the consistency of $S$, and hence $S$ does not prove $T$. We show that we can translate this argument into computability theoretic reductions.

\begin{theorem}
  Let $T = \langle \P_i: \theta_i \to \exists Y \eta_i \rangle_{i \in \omega}$ be a computable enumeration of problems, and $\Q: \theta_Q \to \exists Y \eta_Q$ be an arithmetically definable problem such that each $\P_i$ and $\Q$ are true in the intended model. Here, $\Q$ is arithmetically definable means that both of $\theta_{Q}$ and $\eta_Q$ are arithmetical.
  Assume there are sets $A$ and $B \leq_{\a} A$ such that 
  \begin{itemize}
    \item $B \in \dom \Q$
    \item for any $Y \in \Q(B)$, there exists $\M \leq_{\a} A \oplus Y$ such that $A \in \M$ and $\M \models T$.
  \end{itemize}
  Then, $\Q \not \leq^{\a}_{\omega} T = \{\P_i : i \in \omega\}$ and hence $\Q \not \leq_{\W}^{\a} \bigsqcup T$. Here, $\bigsqcup T$ is the following problem:
  \begin{description}
    \item [Input] A number $i$ and an input $X$ for $\P_i$.
    \item [Output] An output for $\P_i(X)$.
  \end{description}
\end{theorem}
\begin{proof}
  For the sake of contradiction, assume $\Q \leq^{\a}_{\omega} T$.
  Let $A$ and $A_{\Q}$ be as assumed.
  Take an $\omega$-model $\M$ of  $T + \ACAo$ such that $A \in \M$ and no $\M' \in \M$ is an $\omega$-model of $T + \ACAo$ containing $A$.

  We note that $B \in \M$ and $\M \models B \in \dom \Q$. Thus, there exists $Y \in Q(B) \cap \M$.
  For this $Y$, we can find an $\M' \leq_{\a} A \oplus Y$ such that $A \in \M'$ and $\M' \models T + \ACAo$. Then $\M' \in \M$ because $A, Y \in \M$. However, this is a contradiction.
\end{proof}

\begin{remark}
  The above argument is used in \cite{SuzukiYokoyama}.
\end{remark}

  \bibliographystyle{plain}
\bibliography{references}

\end{document}